\documentclass{amsart}
\usepackage{xcolor}
\usepackage{hyperref}
\hypersetup{
    colorlinks=true,
    linkcolor=blue, 
    citecolor=blue,
    urlcolor=magenta,
}

\newtheorem{theorem}{\bf{Theorem}}[section]

\newtheorem{corollary}[theorem]{\bf {Corollary}}
\newtheorem{proposition}[theorem]{\bf {Proposition}}
\newtheorem{lemma}[theorem]{\bf {Lemma}}

\theoremstyle{remark}

\theoremstyle{definition}
\newtheorem{definition}[theorem]{\bf {Definition}}

\usepackage{physics}

\newcommand{\bs}[1]{\Big\{#1  \Big\}} 
\newcommand{\bp}[1]{\Big(#1  \Big)} 

\newcommand{\PP}{{\mathbb{P}}}

\newcommand{\slr}{\mathrm{SL}(d,\Bbb{R})}

\newcommand{\dist}{\mathrm{dist}}
\DeclareMathOperator{\diag}{diag}

\title[The Poisson boundary of semisimple Lie groups]{The Poisson boundary of discrete subgroups of semisimple Lie groups without moment conditions}
\author{K. Chawla}
\author{B. Forghani}
\author{J. Frisch}
\author{G. Tiozzo}
\address{Princeton University, USA}
\email{kc7106@princeton.edu}
\address{The College of Charleston, USA} 
\email{forghanib@cofc.edu}
\address{University of California at San Diego, USA}
\email{joshfrisch@gmail.com}
\address{University of Toronto, Canada}
\email{tiozzo@math.utoronto.ca}

\begin{document}

\maketitle

\begin{abstract}
We consider random walks of finite entropy on Zariski-dense discrete subgroups of semisimple Lie groups 
and show that their Poisson boundary is identified with the Furstenberg boundary of the corresponding symmetric spaces equipped with the hitting measure. We do not assume any moment condition on the random walk and, in contrast with the previous rank $1$ proof, we do not rely on pivoting theory. 
 \end{abstract}

\section{Introduction} 

The classical \emph{Poisson representation formula} establishes a duality between bounded harmonic functions 
on the Poincar\'e disk $\mathbb{D}$ and bounded, measurable functions on its boundary $\partial \mathbb{D} = S^1$. 
This formula is deeply connected with the geometry of $G = SL_2(\mathbb{R})$, which is the group of automorphisms of $\mathbb{D}$. 

In the 1960s, Furstenberg \cite{Furstenberg-semi} (building on Feller \cite{Feller}, Blackwell \cite{Blackwell1955} and Doob \cite{Doob59}) generalized this theory to other locally compact (in particular, Lie) groups. 
In general, given a locally compact, second countable group $G$ and a spread out measure $\mu$ on it, he showed that there is a measure space $(B, \nu)$ such that a generalization of the Poisson representation formula holds; namely, the \emph{Poisson transform} 
\begin{align*}
L^\infty(B, \nu) & \rightarrow   H^\infty(G, \mu) \\
f  & \mapsto  \varphi(g) := \int_B f  \ d g \nu
\end{align*}
is an isomorphism between the space $L^\infty(B, \nu)$ of bounded, measurable functions on the boundary and the space $H^\infty(G, \mu)$ 
of bounded, $\mu$-harmonic functions on $G$.
We call the space $(B, \nu)$ the \emph{Poisson boundary} of the pair $(G, \mu)$.

Probabilistically, the Poisson boundary can also be seen as the space that naturally encodes all possible asymptotic behaviours of the random walk $w_n := g_1 \dots g_n$,
where the sequence $(g_n)$ is independent and each $g_n$ has law $\mu$, and it can be equivalently defined 
as the space of the ergodic components of the shift map in the path space of the random walk.

Poisson boundaries capture algebraic and analytic properties of a group, for instance a countable group is amenable if and only if it admits 
a measure with trivial Poisson boundary \cite{Kaimanovich-Vershik, Rosenblatt81}. 
In many cases the group $G$ acts by isometries on a metric space $(X, d)$, the space $X$ is equipped with 
a natural topological boundary $\partial X$, and one of the central questions in the field has been the ``identification problem", that is whether the  Poisson boundary coincides with $\partial X$, in the sense that the isomorphism above is realized by setting $B = \partial X$ and $\nu$ to be the hitting measure (see e.g. \cite[Sections 2.4, 2.8]{kaimanovich228boundaries}, \cite[Section 1]{Tianyi-ICM}). 

 If $G$ is a semisimple Lie group and $\mu$ is absolutely continuous with respect to the Haar measure on $G$, 
  then Furstenberg \cite{Furstenberg-semi}  showed that the Poisson boundary of $(G, \mu)$ can be identified with the space $B = G/P$, where $P$ is a minimal parabolic subgroup,  which is known as the \emph{Furstenberg boundary}. See \cite[Sections 2.4-2.6]{Furman} 
 and \cite[Sections 2.6, 6.2]{Babillot} for a survey.

On the other hand, if the support of $\mu$ is countable, these results and techniques do not directly apply; 
as observed in \cite[page 149]{kaimanovich228boundaries}, this is because the automorphism group of the corresponding 
Markov operator is not large enough to act transitively on the candidate Poisson boundary. 

Let now $\mu$ be a countably supported measure on a semisimple Lie group $G$, and suppose 
that the semigroup generated by the support of $\mu$ is a discrete subgroup $\Gamma$. 
If the measure $\mu$ has finite first moment, then Furstenberg \cite{Furstenberg63} showed that almost every sample path 
converges in the boundary, and the drift exists and is positive. 
In this case, Ledrappier \cite{Ledrappier-Lie} and Kaimanovich \cite{Kaimanovich-maximal85} showed that the Poisson boundary of the discrete group, i.e. of 
$(\Gamma, \mu)$, is identified with the Furstenberg boundary $G/P$ of the ambient Lie group $G$. 
This identification allows one to compare Lie groups with the lattices contained in them, yielding 
e.g. rigidity statements, such as that certain abstract groups cannot be lattices in certain Lie groups (\cite{Furstenberg-rigidity}, \cite{Bader-Furman}). 

Later, Kaimanovich  \cite{Ka} extended this identification result, between 
the Poisson boundary of the discrete group and the Furstenberg boundary 
of the ambient Lie group, to the more general case that $\mu$ has finite entropy and finite logarithmic moment. In order to do so, he devised two general criteria 
to prove the identification of the Poisson boundary with the geometric boundary, known as the  \emph{ray approximation} and the \emph{strip approximation}. These criteria have been applied in a wealth of contexts, 
especially to groups acting on spaces whose geometry is, in various ways, nonpositively curved: a few examples are 
\cite{
Ballmann-Ledrappier, 
Kaimanovich-hyp94,
Cartwright-Kaimanovich-Woess, Kaim-Masur, KM, 
Kaimanovich-Woess,  Frederic, Horbez, Maher-Tiozzo}, among others. 
However, these criteria require some moment condition on the measure $\mu$ (finite first moment for ray approximation, 
finite logarithmic moment for the strip approximation) to control the geometry of the random walk.

In this paper, we drop every moment condition, and identify the Poisson boundary for random walks on discrete subgroups of semisimple Lie groups
assuming just that the entropy is finite.
Note that if $\mu$ has finite first moment, then the random walk almost surely converges in the visual compactification of the symmetric space $S = G/K$,  where $K$ is a maximal compact subgroup. On the other hand, if one drops the moment condition, the random walk still converges, 
but in a weaker sense, i.e. in the Furstenberg-Satake compactification, whose boundary is $G/P$  \cite{Guivarch-Raugi85}. 
Yet, this is sufficient to define a boundary map, hence it makes sense to ask whether the resulting topological boundary is the Poisson boundary. 

Our main theorem is as follows.

\begin{theorem} \label{T:main-intro}
Let $G$ be a semisimple, connected, Lie group with finite center  and no compact factors, and let $\Gamma < G$ be a discrete, Zariski-dense subgroup. 
Let $\mu$ be a probability measure on $\Gamma$ with finite entropy, such that its support generates $\Gamma$ as a semigroup.  
Then the Furstenberg boundary $(G/P, \nu)$ with the hitting measure $\nu$ is the Poisson boundary for $(\Gamma, \mu)$. 
\end{theorem}

In fact, the condition that $\Gamma$ is Zariski-dense may be relaxed a bit: see Theorem \ref{T:main} for the most general statement.

We note that finite entropy is a necessary assumption for our result and not just a technical artifact. 
Poisson boundaries exhibit radically different phenomena in the presence of finite entropy as opposed to the general case; for example, in infinite entropy there are measures $\mu$ with trivial boundary such that $\mu^{-1}$  has a nontrivial boundary (\cite{Kaimanovich83-examples}, \cite {Alpeev2021examples}), examples of measures on products of groups $G\times H$ whose boundary is nontrivial despite the quotient boundaries on $G$ and $H$ being trivial (\cite{Kaimanovich2024}, \cite{alpeev2024secret}) and, most relevantly for the current project, it is in general impossible to identify a universal space as the Poisson boundary for any group of superpolynomial growth \cite[Theorem 1.4]{Chawla-Frisch}. In particular our result is false in general if we do not include the finite entropy assumption, as there are measures $\mu$, for instance on discrete subgroups of $\textup{SL}(2, \mathbb{R})$, for which 
the Poisson boundary is larger than the geometric boundary \cite{Chawla-Frisch}.


Theorem~\ref{T:main-intro} is the higher rank analog of the main theorem of  \cite{Chawla-Forghani-Frisch-Tiozzo22p}, that for hyperbolic groups, and more generally acylindrically hyperbolic groups,  establish the identification of the Poisson boundary with the Gromov boundary under finite entropy and without moment conditions. 
 
In the footsteps of \cite{Chawla-Forghani-Frisch-Tiozzo22p}(and also \cite{frisch2023poisson}), we replace the \emph{strip} \emph{approximation} from \cite{Ka} by what we call the \emph{pin-down} \emph{approximation}: 
 namely, for any $\epsilon > 0$ one identifies a sequence $(p_n)$ of partitions  of the path space, which can be interpreted as revealing some additional information on the walk, such that: 
 \begin{enumerate}
 \item if $(B, \nu)$ is a $\mu$-boundary, revealing both the boundary point $\xi \in B$ and the outcome of $p_n$ determines, or ``pins down", the location $w_n$, with a subexponential error: this shows that the conditional entropy satisfies $H_B(w_n \vert p_n) = o(n)$;
 
 \item the total information contained in revealing $p_n$ is still small: $H(p_n) \leq \epsilon n$. 
 \end{enumerate} 
 These two facts together imply that the entropy of the random walk conditional to the boundary is sublinear, thus showing that the desired boundary is the Poisson boundary. 
 
 Then, the challenge becomes to construct carefully the partitions $(p_n)$, by revealing certain information on the random walk (e.g., the distance from the origin), and at this step new techniques are required. 
 

In our previous work \cite{Chawla-Forghani-Frisch-Tiozzo22p}, we made crucial use of the  \emph{pivoting theory} of Gou\"ezel \cite{pivot}.
 In fact, in order to construct the partitions $(p_n)$ we kept track of certain times, called \emph{pivotal times}, when the random walk is aligned with the limiting geodesic
 in a fairly strong sense. 
The existence and abundance of such times is there guaranteed by \cite{pivot}, 
that shows strong exponential bounds on the probability of finding such pivots: 
such quantitative estimates were heavily used in \cite{Chawla-Forghani-Frisch-Tiozzo22p}. 
Let us note that an analogue to pivoting theory in higher rank Lie groups has been established in the recent exciting work \cite{Peneau}, but 
it is not clear if it can be applied, and we do not use it in this paper. In fact, our main theorem (Theorem~\ref{T:main-intro}) is stated there as a conjecture (\cite[Conjecture 1.12]{Peneau}). 

In the current paper, instead, we take a different route, avoiding pivoting theory and the corresponding exponential bounds altogether. We replace pivotal times by all times where the random walk is within some bounded distance of the flat; abundance of such times, that we call \emph{critical times}, is guaranteed by a simple recurrence argument using the ergodicity of the shift, without the need for the strong exponential bounds given by pivoting theory (see Lemma \ref{lem:critical}). Let us note that this softer approach would also provide a simpler proof in the rank 1 case.


Moreover, the second difficulty is that in higher rank one needs to replace the limiting geodesic ray by a flat, and one cannot, in contrast with the rank one case, associate a flat to just one boundary point. Therefore, we need to run the random walk both forwards and backwards, and then consider the flat joining the boundary points of the two walks. The pin down partition is then constructed by replacing the distance along a geodesic ray 
by the projection of the random walk onto the flats, using the Cartan projection. 
The fact that flats are translates of abelian groups and have polynomial growth then makes it possible to estimate the conditional entropy, 
concluding the argument. 


\subsection*{Acknowledgements}

B.F. is partially supported by NSF grant DMS-2246727, J.F. is partially supported by NSF grant DMS-2348981 and G.T. is partially supported by NSERC grant RGPIN-2024-04324. We thank Vadim Kaimanovich and Keivan Mallahi-Karai for useful conversations. 

\section{Semisimple Lie groups} \label{S:ss}

Let us start by recalling some fundamental definitions about Lie groups; for details, see e.g. \cite[Section 1]{Guivarch-Raugi85}.

Let $G$ be a semisimple, connected, Lie group with finite center, let $\mathfrak{g}$ be the Lie algebra of $G$, and let $\mathfrak{a}$ be a Cartan subalgebra, with associated Cartan subgroup $A < G$. 

The symmetric space associated to $G$ is the quotient $S = G/K$, where $K$ is a maximal compact subgroup. We will take as a base point of $S$ 
the coset corresponding to $K$, and denote it as $o$. 

When $G = \textup{SL}(d, \mathbb{R})$, we have that $\mathfrak{g}$ is the set of matrices with zero trace, $\mathfrak{a}$ is the set 
of diagonal matrices with zero trace, and $A$ is the group of diagonal matrices with determinant $1$ and positive entries on the diagonal, hence $A \cong \mathbb{R}^{d-1}$. 

A \emph{root} $\alpha$ is a linear map $\alpha : \mathfrak{a} \to \mathbb{R}$ such that the eigenspace
$$\mathfrak{g}_\alpha := \{X \in \mathfrak{g} \ : \  [X, H] = \alpha(H) X \ \forall H \in \mathfrak{a} \}$$
contains a non-zero vector. Let $\Delta$ denote the set of roots, so that we have the decomposition $\mathfrak{g} = \bigoplus_{\alpha \in \Delta} \mathfrak{g}_\alpha$.

A \emph{Weyl chamber} is a connected component of the subset $\mathfrak{a}' \subseteq \mathfrak{a}$ where no non-trivial root vanishes.
Fix a  Weyl chamber $\mathfrak{a}^+ \subseteq \mathfrak{a}$, and let us denote $A^+ := \exp(\mathfrak{a}^+) < G$. Moreover, a root $\alpha$ is called \emph{positive} if $\alpha(H) > 0$ for any $H \in \mathfrak{a}^+$, 
and \emph{negative} if $\alpha(H) < 0$ for any $H \in \mathfrak{a}^+$.

For $G = \textup{SL}(d, \mathbb{R})$, we have that $K = \textup{SO}(d, \mathbb{R})$ and one can take as $\mathfrak{a}^+$ the set of diagonal matrices 
with trace zero and strictly decreasing diagonal entries; then $A^+$ is the set of diagonal matrices with positive diagonal entries in strictly decreasing order and determinant $1$.  

Let 
$$\mathfrak{n} := \bigoplus_{\alpha < 0} \mathfrak{g}_\alpha, \qquad \tilde{\mathfrak{n}} := \bigoplus_{\alpha > 0} \mathfrak{g}_\alpha$$
and $N$, $\tilde{N}$ be the corresponding connected Lie subgroups of $G$. 
For $G = \textup{SL}(d, \mathbb{R})$, $N$ is the subgroup of upper triangular matrices with $1$s on the diagonal, and $\tilde{N}$ is 
the subgroup of lower triangular matrices with $1$s on the diagonal. 

Let $M$ be the centralizer of $A$ in $K$, i.e. $M = \{ m \in K \ : \ m a m^{-1} = a \ \forall a \in A\}$, and let $M'$ be the normalizer of $A$ in $K$, i.e. $M' = \{ m \in K \ : \ m A m^{-1} = A \}$.
The group $W = M'/M$ is a finite group called the \emph{Weyl group}. 
Moreover, if $G = \textup{SL}(d, \mathbb{R})$,  $M$ is the subgroup of diagonal matrices with $\pm 1$ on the diagonal, while $W$ can be thought of as the group 
of permutation matrices, hence $W \cong S_d$, and $M' = M W$.

\subsection{The polar and Bruhat decompositions}

For $G = \textup{SL}(d, \mathbb{R})$, we have the following well-known decomposition. 

\begin{lemma}[\bf Polar decomposition $G=KAK$]
For $g \in\slr$ there exist  orthogonal matrices $k_1, k_2$ and a unique $a=\diag(a_1,\cdots,a_d)\in A$ such that 
$a_1\geq a_2 \geq \cdots \geq a_d>0$ and 
$$
g=k_1 a k_2.
$$ 
Moreover, all other polar decompositions of $g$ are obtained by replacing $(k_1, k_2)$ with $(k_1 m,  m^{-1} k_2)$ for some $m \in M$. 
\end{lemma}

Recall that the \emph{singular values} of a matrix $g$ are the square roots of the eigenvalues of $g^tg$; we denote them as 
$(\sigma_1(g),\dots, \sigma_d(g))$, where we order them so that $\sigma_1(g) \geq \sigma_2(g) \geq \dots \geq \sigma_d(g) \geq 0$.
The entries of $a$ in the polar decomposition of $g$ are the singular values $\{ \sigma_1(g), \dots,  \sigma_d(g) \}$, in some order.

\medskip

In general, let $K < G$ be a maximal compact subgroup, and fix a Weyl chamber $\mathfrak{a}^+ \subseteq \mathfrak{a}$.
The \emph{radial part} of an element $g \in G$ is the unique element $r(g) \in \overline{\mathfrak{a}^+}$ 
such that 
\begin{equation} \label{E:radial}
g = k_1 \exp(r(g)) k_2 \qquad \textup{with }k_1, k_2 \in K.
\end{equation}
As in the linear case, all other such decompositions of $g$ are obtained by replacing $(k_1, k_2)$ with $(k_1 m,  m^{-1} k_2)$ for some $m \in M$. 

\medskip
The following definition \cite[Def. 2.1]{Guivarch-Raugi85}  is essential to guarantee convergence of the random walk. 

\begin{definition}
A sequence $(g_n)_{n \geq 0}$ of elements of $G$ is \emph{contracting} if 
$$\lim_{n \to \infty} \alpha(r(g_n)) = + \infty$$ 
for any positive root $\alpha$. 
Moreover, a semigroup $T < G$ is \emph{contracting} if it contains a contracting sequence. 
\end{definition}

In the case of $G = \textup{SL}(d, \mathbb{R})$, if we let $r_i(g) := \log \sigma_i(g)$ the logarithms of the singular values, a semigroup $T < G$ 
is contracting if 
$$\sup_{g \in T} |r_i(g) - r_{i+1}(g)| = + \infty$$
for any $i = 1, \dots, d-1$.

Let $P = \tilde{N} A M$, which is a maximal amenable subgroup. 
The \emph{Bruhat decomposition} 
$$G = \bigsqcup_{m \in W} N m P$$
is a partition of $G$ in finitely many sets, or \emph{elements}; 
we call \emph{non-degenerate} the element corresponding to $m = e$, which has maximal dimension, 
and \emph{degenerate} all other elements. 

\begin{definition} \cite[Def. 2.5]{Guivarch-Raugi85}
A subgroup $H$ of $G$ is \emph{totally irreducible} if no conjugate of $H$ is contained in the union of finitely many left translates of degenerate elements 
of the Bruhat decomposition. 
\end{definition}

In the case of $G = \textup{SL}(d, \mathbb{R})$, a subgroup $H$ which does not leave invariant any finite union of proper subspaces of $\bigwedge^k \mathbb{R}^d$ for any $k \in \{1, \dots, d-1\}$ is totally irreducible.

\subsection{The Furstenberg boundary} 

\begin{definition}
The \emph{Furstenberg boundary} of the symmetric space $S=G/K$ is the quotient $B=G/ P$. 
\end{definition}

For $G = \textup{SL}(d, \mathbb{R})$, the Furstenberg boundary $G/P$ can be identified as the space of full flags as follows. 
A \emph{full flag} in $\mathbb{R}^d$ is a sequence $V_0 \subset V_1 \subset \dots \subset V_d$ of nested subspaces of $\mathbb{R}^d$, with
$\textup{dim }V_i = i$ for any $0 \leq i \leq d$.  
The \emph{standard flag} is $b^{\uparrow} := (V_i)_{i \leq d}$ with $V_i := \mathbb{R} e_1 \oplus \dots \oplus \mathbb{R} e_i$, while the \emph{opposite standard flag} is $b^{\downarrow} := (W_i)_{i \leq d}$ with $W_i := \mathbb{R} e_{d-i+1} \oplus \dots \oplus \mathbb{R} e_d$.
The group $G = \textup{SL}(d, \mathbb{R})$ acts transitively on the set of full flags, and the stabilizer of the standard flag is the group $P$ 
of upper triangular matrices. Thus, the space of full flags can be identified with $B = G/P$. 

\medskip
By \cite[Lemma 4.1]{Mostow78}, the Furstenberg boundary $B=G/P$ can also be identified as the set of asymptotic classes of Weyl chambers in the symmetric space $S$, where we declare two Weyl chambers to be equivalent if they are within a bounded distance from each other. 

\subsection{Flats and boundary points}

A \emph{flat} in the symmetric space $S = G/K$ is a totally geodesic subspace isometric to $\mathbb{R}^k$ for some $k \geq 1$. 
The \emph{rank} $k_0$ of $S$ is the maximal dimension of a flat; for compatibility with the special linear case, we define $d := k_0 + 1$ so that 
the rank equals $k_0 = d-1$. 
The \emph{standard flat} in $S$ is the orbit $A.o$ of the Cartan subgroup and is diffeomorphic to $\mathbb{R}^{d-1}$. 
Since $G$ acts transitively on the set of maximal flats, each maximal flat is of the form $g A .o$ for some $g \in G$. 

An \emph{oriented flat} is a pair $(f, [w])$ where $f$ is a flat in $S$ and $[w]$ is the asymptotic class of a Weyl chamber $w$ contained in $f$.
Let $\mathcal{F}$ be the set of oriented flats in $S$. 

The \emph{standard oriented flat} is the pair $(A.o, [A^+.o])$. The group $G$ acts transitively on the set of oriented flats, and 
the stabilizer of the standard oriented flat is $AM$, hence the space $\mathcal{F}$ of oriented flats can be identified with $G/AM$.

\medskip
The product $B \times B$ is stratified in $G$-orbits. For any $w \in W$, let us denote as $\mathcal{O}_w \subseteq B \times B$ the $G$-orbit of $(P, wP)$.
Let $w_0 \in W$ be the involution that inverts the orientation of the standard flat.
The only orbit of maximal dimension is $\mathcal{O}_{w_0} = G.(b^{\uparrow}, b^{\downarrow})$, which coincides, in the case of $\textup{SL}(d, \mathbb{R})$, with the set of transverse flags. 
Two full flags $b_1 = (E_i)_{i \leq d}$ and $b_2 = (F_i)_{i \leq d}$ are \emph{transverse} if $E_i \cap F_{d-i} = \{0 \}$ for each $0 \leq i \leq d$. 

\medskip
Note moreover that the stabilizer of the pair $(P, w_0 P)$ is $P \cap w_0 P w_0 = AM$, so we can also identify $\mathcal{O}_{w_0} = G/AM$. 
Thus, by sending $(P, w_0 P)$ to $A.o$ and extending the map by $G$-equivariance, we obtain a $G$-equivariant bijection 
\begin{equation} \label{E:flags-to-flats}
\Phi: \mathcal{O}_{w_0} \to \mathcal{F}.
\end{equation}

In the case of $\textup{SL}(d, \mathbb{R})$, we can define the map $\Phi$ as follows. Given two transverse flags $b_+ = (V_i)_{i \leq d}$, $b_- = (W_i)_{i \leq d}$, for any $1 \leq i \leq d$ we set $E_i := V_i \cap W_{d-i+1}$, which is a one-dimensional subspace. 
Then let $g \in G$ such that $g e_i \in E_i \setminus \{0 \}$ for any $1 \leq i \leq d$. Then set $\Phi(b_-, b_+) := gA.o$. 

\subsection{A generalized distance}  

For any $g \in G$, recall that we have defined in Eq. \eqref{E:radial} the \emph{radial part} $r(g) \in \overline{\mathfrak{a}^+} \subseteq \mathfrak{a}$.
Recall that $\mathfrak{a}$ is endowed with a euclidean norm $\Vert \cdot \Vert_2$, induced by the Killing form. 
Using the radial part, we can now define a ``generalized distance" on $G/K$, with values in the Cartan subalgebra, 
and see some of its basic properties. 

\begin{lemma} \label{D-properties}
Let us define $D : G/K \times G/K \to \mathfrak{a}$ as
$$D(gK, hK) := r(g^{-1} h).$$
This function is well-defined, and has the following two properties: 
\begin{enumerate}
\item $G$-invariance:
$D(g g_1, g g_2) = D(g_1, g_2)$
for any $g, g_1, g_2 \in G$.
\item Lipschitz property:
$$\Vert D(g_1, h) - D(g_2, h) \Vert_2 \leq \Vert D(g_1, g_2) \Vert_2 $$
for any $g_1, g_2, h \in G$.
Similarly, 
$$\Vert D(h, g_1) - D(h, g_2) \Vert_2 \leq   \Vert D(g_1, g_2) \Vert_2$$
for any $g_1, g_2, h \in G$.
\end{enumerate}
\end{lemma}

\begin{proof}
To show that the function is well-defined on $G/K \times G/K$, note that, if $g_1 = g k_1$ and $h_1 = h k_2$ with $k_1, k_2 \in K$, then 
$$r(g_1^{-1} h_1) = r(k_1^{-1} g^{-1} h k_2) = r(g^{-1} h).$$ 
Now, part (1) is clear by definition. 

Part (2) is also well-known, for a proof see e.g. \cite[Lemma 2.3]{Kassel}. 

\end{proof}

Note that the Riemannian distance $\dist$ on the symmetric space $S = G/K$ is 
\begin{equation} \label{E:distance}
\dist(g_1 K, g_2 K) = \Vert D(g_1 K, g_2 K) \Vert_2 = \Vert r(g_1^{-1} g_2) \Vert_2.
\end{equation}

\section{Random walk, Poisson boundary and Entropy} 

Let $\Gamma$ be a countable group equipped with a probability measure $\mu$. 
Given a probability measure $\theta$ on $\Gamma$, we define the \emph{random walk} driven by $\mu$ with initial distribution $\theta$ 
as the process $(w_n)_{n \geq 0}$ defined as 
$$
w_n := g_0g_1\cdots g_n,
$$
where $(g_n)_{n \geq 0}$ is a sequence of independent random variables, $g_0$ has distribution $\theta$ and each $g_n$ for $n \geq 1$ 
has distribution $\mu$. 
We call a \emph{sample path} an infinite sequence $\omega = (w_n)_{n \geq 0}$ and we denote by $\Omega$ the space of such infinite sequences, and by $\PP_\theta$ the corresponding measure on $\Omega$. 

The sequence $(g_n)_{n\geq 1}$ is called the sequence of \emph{increments} of the sample path $\omega =(w_n)_{n\geq 0}$. When $\theta$ is concentrated on the identity element $e$ of $\Gamma$, that is $\theta=\delta_e$, we write $\PP=\PP_{\delta_e}$.
For a sample path $\omega =(w_n)_{n\geq0}$, we define $U(\omega) := (w_1^{-1}w_{n+1})_{n\geq1}$ as the \emph{shift in increments}. Consequently, the $i^\text{th}$-iterate of the shift in increments is $U^i(\omega)=(w_i^{-1}w_{n+i})_{n\geq 1}$. Note that $U$ is measure-preserving and ergodic. 

\subsection{Poisson boundary} 

Let $\mu$ and $\theta$ be two probability measures on $\Gamma$ such that $\theta(g)>0$ for every $g$ in $\Gamma$. Consider the space of sample paths $(\Omega,\PP_{\theta})$. Two sample paths $(w_n)_{n\geq0}$ and $(w'_n)_{n\geq0}$ are equivalent when there exist $k$ and $k'$ such that $w_{n+k}=w'_{n+k'}$ for all $n\geq 0$.  Denote by $\mathcal{I}$ the sigma-algebra generated by all measurable unions of these equivalence 
classes (mod $0$) with respect to $\PP_{\theta}$.  Thus by Rokhlin's theory of Lebesgue spaces \cite[No.2, p.30]{Rokhlin52},  there exist a unique (up to isomorphism) Lebesgue space $\partial_\mu \Gamma$ equipped with a sigma-algebra $\mathcal{S}$ and a measurable function 
$$
\mbox{bnd}: \Omega \to \partial_\mu \Gamma
$$
such that the pre-image of $\mathcal{I}$ under the map is $\mathcal{S}$. Let $\nu$ be the image of $\PP$ under map $\mbox{bnd}$. The probability space $(\partial_\mu \Gamma, \nu)$ is called the \emph{Poisson boundary} of the $(\Gamma,\mu)$ random walk.

Because the semigroup generated by $\mu$ acts on sample paths by $g.(w_n)_{n\geq0}=(gw_n)_{n\geq0}$, this action extends to an action on the Poisson boundary.  Moreover, $\nu$ is $\mu$-stationary, that is
$$
\nu = \sum_{g\in \Gamma} \mu(g) g\nu.
$$
 A quotient of the Poisson boundary with respect to a $\Gamma$-equivariant  partition is called a $\mu$-boundary. Thus, the Poisson boundary is the maximal $\mu$-boundary.  
 
 Note that throughout this paper, a partition is a measurable partition in the sense of Rokhlin \cite[Section I.3]{Rohlin67}.

\subsection{Entropy}  We will use the language of partitions to formulate entropy. 

Given a partition $\gamma$ on the space of sample paths and $\omega \in \Omega$, let $\gamma[\omega]$ denote the class that includes $\omega$.
We denote the (Shannon) entropy of the partition $\gamma$ by 
$$
H_\PP(\gamma)=H(\gamma) =- \int_\Omega \log\PP(\gamma[ \omega]) \ d\PP(\omega).
$$

Given a random variable $Y : \Omega \to \Sigma$ with values in a countable set $\Sigma$, we define the preimage partition
$\gamma_Y := \bigsqcup_{y \in \Sigma}\{ \omega \in \Omega \ : \ Y(\omega) = y\}$
and 
$$H(Y) = H(\gamma_Y) = - \sum_{y \in \Sigma} \log\PP(Y(\omega) = y) \ \PP(Y(\omega) = y).$$

Suppose that $\gamma$ and $\beta$ are two countable partitions on  $(\Omega,\PP)$. The joint partition $\gamma \vee \beta $ of $\gamma$ and $\beta$ is defined by setting for every $\omega \in \Omega$
$$
 (\gamma \vee \beta)[\omega] = \gamma[\omega] \cap \beta[\omega].
$$
By the properties of entropy, one can show the following.
\begin{lemma}\label{lem:joint}
Let $\gamma$ and $\beta$ be two countable partitions. Then, 
\begin{enumerate}
\item $H(\gamma \vee \beta)  \leq H(\gamma) + H(\beta).$
\item If the cardinality of $\gamma$ is  $|\gamma|$,  then $H(\gamma) \leq \log |\gamma|$.
\end{enumerate}
\end{lemma}

Let $E$ be a measurable set in $\Omega$. For a countable partition $\gamma$, we define the partition $\gamma^E$ such that 
$$
\gamma^E[\omega] = \begin{cases}
\gamma[\omega] \cap E & \omega\in E\\
E^c=\Omega - E & \omega \not \in E
\end{cases}
$$
We need the following lemma, which is an application of uniform integrability of $L^1$ functions, see \cite[Lemma, 2.4]{Chawla-Forghani-Frisch-Tiozzo22p}.

\begin{lemma}\label{lem:uniform}
Let $\gamma$ be a countable measurable partition on $(\Omega, \PP)$ with finite entropy. Then for every $\epsilon>0$ there exists $\delta>0$ such that for every measurable set $E$ with $\PP(E) < \delta$,
$$
H(\gamma^E) <\epsilon.
$$
\end{lemma}

\subsection{Conditional Entropy} 
Let $(X,\lambda)$ be a $\mu$-boundary. Then, for $\lambda$-almost every point $\xi \in X$ a system of conditional measures $\{\PP^\xi\}_{\xi\in X}$ exists such that 
$$
\PP = \int_X \PP^\xi \  d\lambda(\xi).
$$
We denote the conditional entropy given $\xi \in X$ by
$$ 
H_\xi(\gamma)=H_{\PP^\xi}(\gamma)= -\int_\Omega \log\PP^\xi(\gamma[\omega])\  d\PP^\xi(\omega) 
$$  
and the conditional entropy of the $\mu$-boundary $(X,\lambda)$ by 
$$
H_X(\gamma) = \int_X H_\xi(\gamma) \ d\lambda(\xi).
$$
Let $\eta_X$ be the associated partition to the $\mu$-boundary $(X,\lambda)$, thus two sample paths are equivalent when they  have the same boundary point in $X$. Alternative notations include $H(\gamma|\xi)=H_\xi(\gamma)$ and $H_X(\gamma) = H(\gamma | \eta_X)$.

Denote by $\alpha_n$ the partition on the space of sample paths such that two sample paths are $\alpha_n$-equivalent when they have the same $n^\text{th}$-step. In this case,  
$$
H(\alpha_n) = -\sum_g \mu^{*n}(g) \log\mu^{*n}(g),
$$
where $\mu^{*n}$ is the $n^\text{th}$-fold convolution of $\mu$.

We say $\mu$ has finite entropy when $H(\alpha_1)$ is finite. 
One can show that the sequence $\{H(\alpha_n)\}_{n\geq1}$ is subadditive, and the \emph{asymptotic entropy} (also known as the \emph{Avez entropy}) of the $\mu$-random walk   is defined as  
$$
h(\mu) = \lim_{n\to\infty} \frac{H(\alpha_n)}{n}.
$$
Note that when  $(X,\lambda)$ is a $\mu$-boundary,  the Furstenberg entropy is defined as
$$
h_\mu(X,\lambda) = \sum_g \mu(g) \int_X \log\frac{dg\lambda}{d\lambda}(\xi) \ dg\lambda(\xi).
$$
Kaimanovich-Vershik \cite[Theorem 3.2]{Kaimanovich-Vershik} and Derriennic \cite[Th\'eor\`eme, p. 268]{De} proved that 
$$h_{\mu}(X,\lambda) \leq h(\mu).$$ 
Moreover, when $\mu$ has finite entropy, the equality holds if and only if $(X,\lambda)$ is the Poisson boundary.  
We use the following entropy criterion to determine whether a $\mu$-boundary is the Poisson boundary.

\begin{theorem}\cite[Theorem 2]{Kaimanovich-maximal85}\label{thm:maximal}
Let $(X,\lambda)$ be a $\mu$-boundary. If $\mu$ has finite entropy, then 
$$
h_X=\lim_{n\to\infty} \frac{H_X(\alpha_n)}{n} $$
exists.  Moreover, $(X,\lambda)$ is the Poisson boundary if and only if  $h_X=0$.
\end{theorem}

\section{Random walks on semisimple Lie groups}

\begin{definition}
A measure $\nu$ on $B = G/P$ is \emph{irreducible} if $g \nu(NmP) = 0$ for any $g \in G$, $m \in W \setminus \{ \overline{e} \}$.
\end{definition}

Let $G_\mu$ be the closed subgroup generated by the support of $\mu$, and let $T_\mu$ be the closed subsemigroup generated by the support of $\mu$.

\begin{theorem}\cite[Thm 2.6]{Guivarch-Raugi85} \label{T:converge}
Let $\mu$ be a probability measure on a semisimple, connected, Lie group $G$ with finite center. 
Suppose that $T_\mu$ contains a contracting sequence and that $G_\mu$ is totally irreducible. 

Then there exists a unique $\mu$-stationary probability measure $\nu$ on $B = G/P$ and this measure is irreducible. 

Moreover, there exists a $B$-valued random variable $Z$ such that the sequence of measures $(g_1 g_2 \dots g_n \nu)$ converges almost surely to the Dirac measure $\delta_{Z(\omega)}$. 
\end{theorem}

Let us also show that limit points of the random walk are almost surely pairs of transverse flags. 

\begin{corollary} 
For any $b_- \in B$, we have $\nu(\{ b_+ \in B \ : \ (b_-, b_+) \in \mathcal{O}_{w_0} \} ) = 1$.
\end{corollary}

\begin{proof}
Let $w \in W$. Let $g_1 \in G$ such that $b_- = g_1P$, and let $g \in G$ such that $b_+ = gP$.
Then $(b_-, b_+) = (g_1 P, g P)$ belongs to $\mathcal{O}_w$ if and only if there exists $h \in G$ such that $g_1 P = hP$, $gP = hwP$. 
Hence $h \in g_1 P$, so $g P \in g_1 P wP$. Now, recall that $P = \tilde{N} A M$, $\tilde{N} = w_0 N w_0$, so 
$$g_1 P w P = g_1 \tilde{N} A M w P = g_1 w_0 N w_0 AM w P = g_1 w_0 N w_0 w AM P =  g_1 w_0 N w_0 w P$$
By the definition of irreducibility, 
$$\nu(\{ b_+ \in B \ : \ (b_-, b_+) \in \mathcal{O}_{w} \} ) = \nu(g_1 P w P) = \nu(g_1 w_0 N w_0 w P) = 0$$
unless $w_0 w = \overline{e} \in W$, hence $w = w_0$. Since the sets $\mathcal{O}_w$ for distinct $w \in W$ are disjoint, 
the claim follows.
\end{proof}

\begin{definition}
Let $G$ be a connected, semisimple Lie group with finite center, and let $\mu$ be a probability measure on $G$ with countable support. 
We say that $\mu$ is \emph{totally irreducible and bi-contracting} if $G_\mu$ is totally irreducible and both $T_\mu$ and $T_\mu^{-1}$ contain a contracting sequence.
\end{definition} 

In particular,  if the semigroup generated by $\mu$ is Zariski dense, then the probability measure $\mu$ is totally irreducible and bi-contracting by \cite{Guivarch-Raugi89}. 

Let now $\mu$ be a totally irreducible, bi-contracting measure $\mu$ on $G$. 
By the above theorem, there exists a unique $\mu$-stationary probability measure $\nu$ on $B$, 
and a $B$-valued random variable $Z$ such that 
$$g_1 \dots g_n \nu \to \delta_{Z(\omega)}$$
for almost every $\omega \in \Omega$.

Since from now on we will also deal with bilateral random walks, let us consider the space of \emph{bilateral increments} $(G^\mathbb{Z}, \mu^{\otimes \mathbb{Z}})$, whose elements we denote as $(g_n)_{n \in \mathbb{Z}}$.
Let now $\overline{\Omega}$ denote the space of \emph{bilateral sample paths}: its elements are also bi-infinite sequences of elements of $G$, and are denoted as $\omega = (w_n)_{n \in \mathbb{Z}}$, where 
$$w_n := \left\{ \begin{array}{ll} 
g_1 g_2 \dots g_n & \textup{if }n > 0\\
e & \textup{if }n = 0\\
g_{0}^{-1} g_{-1}^{-1} \dots g_{n+1}^{-1} & \textup{if }n < 0.
\end{array} \right.$$
We denote as $\overline{\mathbb{P}}$ the induced probability measure on $\overline{\Omega}$, so that $(\overline{\Omega}, \overline{\mathbb{P}})$ 
is the probability space of sample paths for the bilateral random walk. Note that the sequence $(w_{-n})_{n \geq 0}$ follows a random walk on $G$ 
driven by the \emph{reflected measure} $\check{\mu}$, where $\check{\mu}(g) := \mu(g^{-1})$, and independent of $(w_n)_{n \geq 0}$. 
We often call $(w_n)_{n \geq 0}$ the \emph{forward random walk} and $(w_{-n})_{n \geq 0}$ the \emph{backward random walk}. 

Thus, applying Theorem \ref{T:converge} to the backward random walk, 
there exists a unique $\check{\mu}$-stationary probability measure $\check{\nu}$ on $B$, 
and a $B$-valued random variable $\check{Z}$ such that 
$$g_0^{-1} \dots g_{-n}^{-1} \check{\nu} \to \delta_{\check{Z}(\omega)}$$
for almost every $\omega \in \overline{\Omega}$.
Hence, this defines a measurable map 
\begin{equation} \label{E:double-boundary}
(\overline{\Omega}, \overline{\mathbb{P}}) \to (B \times B, \nu \otimes \check{\nu}).
\end{equation}
Finally, the bilateral hitting measure $\nu \otimes \check{\nu}$ is supported on $\mathcal{O}_{w_0} \subseteq B \times B$. 
Moreover, for any pair $(b_-, b_+) \in \mathcal{O}_{w_0}$, there exists a unique oriented flat 
$\Phi(b_-, b_+)$ with endpoints $(b_-, b_+)$. 

Thus, let us define the map $F : \overline{\Omega} \to \mathcal{F}$ as 
$$F(\omega) := \Phi(Z(\omega), \check{Z}(\omega)).$$

\section{The pin-down argument}

\subsection{Critical times}

Fix a constant $\alpha > 0$. For any $k \geq 0$, denote by $I_{k, \alpha}$ the time interval $[k \alpha, (k+1) \alpha) \cap \mathbb{N}$.
In order to bound the conditional entropy of the random walk, we will subdivide the interval $[0,n]$ into $n/\alpha $ subintervals $I_{k, \alpha}$, 
each of length $\alpha$.

\begin{definition}[Critical times]
Let $\omega = (w_i)_{i \in \mathbb{Z}}$ be a bilateral sample path and   $M>0$ and $\alpha > 0$ be fixed constants. 
We call time $i$ \emph{critical} (depending on $M,n,\alpha$ and $\omega$) if $i$ is the first time in its subinterval such that $\dist(w_i. o, F(\omega)) \leq M$, 
meaning that the sample path is close to flat with respect to the Riemannian metric. 
\end{definition}
We will show that critical times occur quite often for a universal $M>0$. 

\begin{lemma}[Plenty of critical times]\label{lem:critical}
Suppose that $\mu$ is a totally irreducible and bi-contracting probability measure on $G$. 
There exists $M > 0$ such that for any $\epsilon > 0$ there exists $k$ such that 
$$\PP\bp{\dist(w_i. o, F(\omega)) \geq M \textup{ for all } i \in [n, n +k] } < \epsilon$$
for any $n$.
\end{lemma}

\begin{proof}
Define the set
$$
A := \bs {\omega \in \overline{\Omega} \ : \ \dist(o, F(\omega)) \geq M },
$$
where $M$ is chosen so that $0 < \PP(A) < 1$. 
Let $U$ be the shift in the space of increments.  Given that $U$ is measure-preserving and ergodic and $0<\PP(A)<1$, 
we obtain
$$\PP\bp{\bigcap_{i = 0}^\infty U^{-i}A}= 0.$$
Since $U$ is measure-preserving, for any $\epsilon > 0$ there exists $k$ such that 
$$\PP\bp{\bigcap_{i = n}^{n+k} U^{-i}A} = \PP\bp{\bigcap_{i = 0}^k U^{-i}A} < \epsilon.$$
Now note that, by $G$-invariance of the distance, 
\begin{align*}
U^{-i}A & = \bs{ \omega \in \overline{\Omega} \ : \ \dist(o, F(U^i \omega) ) \geq M } \\
& = \bs{ \omega \in \overline{\Omega} \ : \ \dist(o, w_i^{-1} F(\omega) ) \geq M } \\
& = \bs{ \omega \in \overline{\Omega} \ : \ \dist(w_i. o,  F(\omega)) \geq M },
\end{align*}
hence the claim follows.
\end{proof}

\begin{definition} \label{def:good-times}
Let $n>0$ be an integer, $\alpha>0,L>0$. We fix an $M>0$  as in Lemma~\ref{lem:critical}. We say that an interval $I_{k,\alpha}$ is \emph{$L$-good} 
for $1\leq k < \frac{n}{\alpha}$ when 
\begin{enumerate}
\item   there exists a critical time in $I_{k,\alpha}$,
\item all step (increment) sizes within $I_{k,\alpha}$  are at most  $L$: 
$$
\dist(w_i.o, w_{i+1}.o) \leq L \hspace{1cm} \forall i\in I_{k,\alpha}.
$$
\end{enumerate}
Otherwise, we say the interval $I_{k,\alpha}$ is \emph{$L$-bad}.

Moreover, by definition we declare both the first interval $I_{0, \alpha}$ and the last interval $I_{\lfloor n/\alpha \rfloor, \alpha}$ 
to be $L$-bad. 
\end{definition}

\subsection{Defining the partitions}
Let us fix a pair $(b_-, b_+)$ of transverse flags in $G/P \times G/P$, which we think of as the two boundary points of, respectively, the backward and forward random walk. 
As we saw earlier in Eq. \eqref{E:flags-to-flats}, this choice determines an oriented flat $F$ in the symmetric space. 

Moreover, let $p \in S$ be the closest point projection of the basepoint $o$ onto $F$. 
Then, there exists $g \in G$ such that $F = g A.o$ and also $p = g.o$.
The choice of $g$ is unique up to multiplication by $M'$ if we consider $F$ as unoriented, and up to $M$ if we take into account the orientation on $F$. 

Let $\log : A \to \mathfrak{a}$ be the inverse of the exponential map, and 
let $\textup{proj}_F : G/K \to F$ be the closest point projection onto $F$. Now, let $\pi_{F} : G/K \to \mathfrak{a}$ be the projection defined as follows: for $x \in G/K$, let $y = \textup{proj}_F(x) \in F$.
Then let $a \in A$ be such that $y = g a.o$, and define $\pi_{F}(x) := \log a$. 

\medskip

Let $0 < k_1 < k_2 < \dots < k_r \leq n$ be the critical times, in order. 
\begin{definition}
We call an index $j$ \emph{doubly good} if $k_j$ and $k_{j+1}$ lie in consecutive good intervals: that is, if there exists $k \leq n/\alpha$ such that $k_j \in I_{k, \alpha}$, $k_{j+1} \in I_{k+1, \alpha}$, and both $I_{k, \alpha}$ and $I_{k+1, \alpha}$ are $L$-good intervals. 
\end{definition}

We denote as $\mathcal{DG} \subseteq \{0, \dots, r\}$ the set of doubly critical indices. 
Now, let us fix once and for all a linear isomorphism $\iota: \mathfrak{a} \to \mathbb{R}^{d-1}$, and given a vector $v \in \mathfrak{a}$, 
we denote as $\lfloor v \rfloor \in \mathbb{Z}^{d-1}$ a choice of closest point to $\iota(v)$ in $\mathbb{Z}^{d-1}$, 
according to the metric induced by $\Vert \cdot \Vert_2$ on $\mathbb{R}^{d-1}$. 

\begin{definition}
Let $n,\alpha$ and $L$ be as before. For a sample path $\omega =(w_i)_{i \in \mathbb{Z}}$, the \emph{good projection}  is defined as 
$$
p_n^{\alpha,L}(\omega):=\left\lfloor \sum_{j \in \mathcal{DG}} \pi_{F}(w_{k_j}.o) - \pi_{F}(w_{k_{j+1}}.o) \right\rfloor.
$$
The sum is over all doubly good indices $j$ with $0 \leq j < \frac{n}{\alpha}$. 
\end{definition}

Let us recall that $W$ denotes the Weyl group, and let us now define a map $\sigma : \mathfrak{a} \to W$ as follows: for any $v \in \mathfrak{a}$, we let $\sigma(v)$ be an element of the Weyl group such that $\sigma(v).v \in \overline{\mathfrak{a}^+}$.
When $G = \textup{SL}(d, \mathbb{R})$, then $W = S_d$ is the group of permutations on the set $\{1, \dots, d \}$, 
and $\sigma(v)$ essentially records the order of the entries of $v \in \mathfrak{a} \subseteq \mathbb{R}^d$. 

Now, we record the information of the random walk at time $n$ via the following procedure, that gives rise to $4$ sets of partitions 
of the path space..

\begin{enumerate}
\item
We define as $\tau_{n}^{\alpha, L}$ the partition associated to recording the sequence $(k_1, \dots, k_r)$ of critical times. 

\item 
We record the value of the good projection $p_{n}^{\alpha, L}$ and denote as $\pi_n^{\alpha, L}$ the associated partition. 

\item
If an interval $I_{k, \alpha}$ is bad, we record all increments in the current interval, as well as the previous and the next one. 
More precisely, if we let $\mathcal{B} : = \{ k \in [0, n/\alpha) \ : \ I_{k, \alpha} \textup{ is bad} \}$
and $J_{k, \alpha}:= I_{k-1, \alpha} \cup I_{k, \alpha} \cup I_{k+1, \alpha}$, we record
$$\big((g_i)_{i \in J_{k, \alpha}}\big)_{k \in \mathcal{B}}$$
and we call the partition associated to this random variable $\beta_n^{\alpha, L}$. 

\item 
For each index $j \in [0, n/\alpha)$ that is not doubly good, we record
$$ \sigma\left( \pi_F(w_{k_{j+1}}.o) - \pi_F(w_{k_j}.o) \right)$$
that is, essentially, the order of the entries of the difference $\pi_F(w_{k_{j+1}}.o) - \pi_F(w_{k_j}.o)$.
We denote the associated partition by $\sigma_n^{\alpha, L}$.

\end{enumerate}

\subsection{Entropy estimates}

\begin{proposition} \label{E:entropy-upper-bound}
For any $\epsilon > 0$ there exists $\alpha_0 > 0$ such that for any $\alpha \geq \alpha_0$ there exists $L >0$ such that 
$$\limsup_{n \to \infty} \frac{1}{n} H(\tau_n^{\alpha, L} \vee \pi_n^{\alpha, L} \vee \sigma_n^{\alpha, L} \vee \beta_n^{\alpha, L}) \leq \frac{\log \alpha}{\alpha} + \frac{\log(\# W)}{\alpha} + \epsilon.$$
As a corollary, for any $\epsilon > 0$ there exist $\alpha, L >0$ such that 
$$\limsup_{n \to \infty} \frac{1}{n} H(\tau_n^{\alpha, L} \vee \pi_n^{\alpha, L} \vee \sigma_n^{\alpha, L} \vee \beta_n^{\alpha, L}) \leq \epsilon.$$
\end{proposition}

\begin{proof}

There are at most $n/\alpha$ critical times and each of them has at most $\alpha$ values. 
Hence the entropy of the set of critical times is bounded by
\begin{equation} \label{E:tau-bound}
H(\tau_{n}^{\alpha, L}) \leq \frac{n}{\alpha} \log \alpha
\end{equation}

Note that, since symmetric spaces of non-compact type are CAT(0), and  closest point projection in CAT(0) spaces is distance non-increasing, 
by definition of $L$-good we have
$$\Vert \pi_{F}(w_{k_j}.o) - \pi_{F}(w_{k_{j+1}}.o) \Vert_2 \leq \dist(w_{k_j}.o, w_{k_{j+1}}.o) \leq L$$ 
so $p_n^{\alpha,L}(\omega)$ is a vector in $\mathbb{Z}^{d-1}$ of length  at most $L n$.

Hence, the entropy of $\pi_n^{\alpha, L}$ is bounded above by 
\begin{equation}\label{eq:entropy-projection}
H(\pi_n^{\alpha,L}) \leq  d \log(nL).
\end{equation}

Since there are at most $n/\alpha$ critical times and there are $\# W$ elements in the Weyl group, 
the entropy of 
$\sigma_{n, L}^{\alpha}$ is bounded by 
\begin{equation} \label{E:sigma-bound}
H(\sigma_{n, L}^{\alpha}) \leq \frac{n}{\alpha} \log(\# W).
\end{equation}

Finally, let us estimate the entropy of $\beta_n^{\alpha, L}$. 
First, given $\epsilon > 0$, let $\delta > 0$ be determined by Lemma \ref{lem:uniform} applied to 
the random variable $g_1$, which has finite entropy since $\mu$ does. 
Now, from Lemma \ref{lem:critical} there exists $\alpha_0$ such that 
for any $\alpha \geq \alpha_0$ 
$$\mathbb{P}(I_{k, \alpha} \textup{ contains no critical time}) \leq \delta$$
for any $0 < k < \lfloor n/\alpha \rfloor$. 
Hence, by choosing $L$ large enough, one can also have condition (2) of Definition \ref{def:good-times} holds, 
hence we obtain 
$$\mathbb{P}(I_{k, \alpha} \textup{ is bad}) \leq \delta/2$$
for any $0 < k < \lfloor n/\alpha \rfloor$. 
Then, by Lemma \ref{lem:uniform} we have  
$$H(g_i \mathbf{1}_{I_{k, \alpha} \textup{ is bad}}) \leq \epsilon$$
for any for any $0 < k < \lfloor n/\alpha \rfloor$ and $i \in I_{k, \alpha}$, hence, 
by also taking into account that the first and last intervals are declared to be bad,  
\begin{equation} \label{E:beta-bound}
H(\beta_n^{\alpha, L}) \leq \epsilon n + 2 \alpha H(g_1).
\end{equation}
The claim follows by combining \eqref{E:tau-bound}, \eqref{eq:entropy-projection}, \eqref{E:sigma-bound}, and \eqref{E:beta-bound}, and taking the limsup as $n \to \infty$.
\end{proof}

\begin{proposition} \label{P:pin-down}
Suppose that $\mu$ is a totally irreducible and bi-contracting probability measure on a discrete subgroup $\Gamma$ of $G$. Then for every $\alpha\geq 1$ and $L>0$, the joint partitions $\tau_n^{\alpha,L}, \pi_n^{\alpha,L}, \sigma_n^{\alpha,L}$, and $\beta_n^{\alpha,L}$ pin down the conditional location of the random walk at time $n$; that is, 
$$
\lim_{n \to \infty} \frac{1}{n} H_{B \times B}(\alpha_n |\tau_n^{\alpha,L} \vee \pi_n^{\alpha,L}  \vee  \sigma_n^{\alpha,L} \vee \beta_n^{\alpha,L} ) = 0.
$$
\end{proposition}

\begin{proof}
We fix a pair of transverse boundary points $(b_-, b_+)$ in the Furstenberg boundary $B=G/P$. Suppose that the bilateral sample path $\omega =(w_i)_{i \in \mathbb{Z}}$ converges to the pair of boundary points, and let $g \in G$ be such that $F = g A.o$ is the associated oriented flat.

\medskip

Let $0 \leq k_1 < \dots < k_r \leq n$ be the critical times, so that $k_r$ is the last critical  time before $n$. 
We claim that given the partitions $\tau_n^{\alpha,L}, \beta_n^{\alpha,L}$, and $\pi_n^{\alpha,L}$ is sufficient to compute  $\pi_{F}(w_{k_r}.o)$, up to an error of at most $n/\alpha$, 
which gives rise to at most $(n/\alpha)^d$ choices for $w_{k_{r}}.o$.

Suppose that $k_j < k_{j+1}$ are two consecutive critical times. If the index $j$ is doubly good, then $k_j, k_{j+1}$ lie in consecutive good intervals, and the value
$$\pi_{F}(w_{k_{j+1}}.o) -  \pi_{F}(w_{k_{j}}.o)$$
is one of the summands of the good projection $p_n^{\alpha,L}$. 

If not, then there exist two elements $a_j, a_{j+1} \in A$ such that $ga_j.o = \textup{proj}_F(w_{k_j}.o)$ and $ga_{j+1}.o = \textup{proj}_F(w_{k_{j+1}}.o)$ on $F$ are within distance $M$, respectively, of $w_{k_j}.o$ and $w_{k_{j+1}}.o$. 
Since all increments between $k_j$ and $k_{j+1}$ are given via  the partition $\beta_n^{\alpha, L},$ and since $D$ is $G$-invariant, we know  
$$
D(w_{k_j}.o, w_{k_{j+1}}.o) = D(o, g_{k_j+1} g_{k_j+2} \dots g_{k_{j+1}}.o).
$$
By Lemma \ref{D-properties} and $G$-invariance,
$$
D(o, a_j^{-1} a_{j+1}.o) = D(ga_j.o, g a_{j+1}.o) = D(w_{k_j}.o, w_{k_{j+1}}.o) + O(1).
$$
Thus we know the radial part of $a_j^{-1} a_{j+1}$ up to a uniform additive error. 
Note that any element $a \in A$ is determined by the pair $(r(a), \sigma(\log a))$, hence setting
$$v_j:= \log(a_j^{-1} a_{j+1}) = \pi_{F}(w_{k_{j+1}}.o) -  \pi_{F}(w_{k_{j}}.o)$$ 
and noting that $\sigma(v_j)$ is given via the partition $\sigma_n^{\alpha,L}$, we obtain that we also know 
the value of
\begin{equation} \label{E:not-doubly-good}
v_j = \pi_{F}(w_{k_{j+1}}.o) -  \pi_{F}(w_{k_{j}}.o)
 \end{equation}
up to uniform additive error. 
Then 
\begin{align*}
\pi_{F}(w_{k_r}.o) - \pi_{F}(w_{k_1}.o) & = \sum_{j=1}^{r-1} \left(\pi_{F}(w_{k_{j+1}}.o) -  \pi_{F}(w_{k_{j}}.o) \right)\\
& = \sum_{j \in \mathcal{DG}} \left(\pi_{F}(w_{k_{j+1}}.o) -  \pi_{F}(w_{k_{j}}.o) \right) + \sum_{j \notin \mathcal{DG}} \left(\pi_{F}(w_{k_{j+1}}.o) -  \pi_{F}(w_{k_{j}}.o) \right)
\end{align*}
and the first term is the good projection, up to $O(1)$, while the second term is the sum of the previous contributions from \eqref{E:not-doubly-good}, each up to an additive error. Since the number of such terms is bounded above by the number of critical times, the error is an additive error of at most $O(n/\alpha)$.
Hence, we know the vector 
$$\lfloor \pi_{F}(w_{k_r}.o) - \pi_{F}(w_{k_1}.o) \rfloor \in \mathbb{Z}^{d-1}$$
up to at most $O((n/\alpha)^{d-1})$ choices. 

Finally, since the first interval is bad by definition, and we record via $\beta_{n}^{\alpha, L}$ all increments up to and including the first good interval, 
we know $w_{k_1}$, hence we also know $\pi_F(w_{k_1}.o)$.

By using the knowledge of $\lfloor \pi_{F}(w_{k_r}.o) - \pi_{F}(w_{k_1}.o) \rfloor$ we now obtain the location of $\pi_{F}(w_{k_r}.o)$ up to $O((n/\alpha)^{d-1})$ choices, and we know that $w_{k_r}.o$ lies within a ball of radius $M$ of $\textup{proj}_{F}(w_{k_r}.o)$.
Hence, by using that the action of $\Gamma$ on the symmetric space is discrete, we reconstruct $w_{k_r} \in \Gamma$ up to $O((n/\alpha)^{d-1})$ choices.

Moreover, since the last interval is bad by definition, and we record via $\beta_{n}^{\alpha, L}$ all increments after the last good interval, we know all increments between $w_{k_r}$ and $w_n$, hence we know $w_{k_r}^{-1} w_n$. 

Thus, we pin down $w_n$ up to $O((n/\alpha)^{d-1})$ choices. By taking the $\log$ and the limit as $n \to \infty$, we obtain the claim.
\end{proof}

Let us now state and prove the main theorem of this paper, in its most general form. 

\begin{theorem} \label{T:main}
Let $G$ be a semisimple, connected, Lie group with finite center, and let $\Gamma < G$ be a discrete subgroup. 
Let $\mu$ be a totally irreducible, bi-contracting probability measure on $\Gamma$, with finite entropy. 
Then the Furstenberg boundary $(G/P, \nu)$ with the hitting measure $\nu$ is the Poisson boundary for $(\Gamma, \mu)$. 
\end{theorem}

Note that the condition on the measure $\mu$ holds if the semigroup generated by the support of $\mu$ is a Zariski-dense subgroup of $G$ (\cite{Goldsheid-Margulis}), immediately yielding Theorem \ref{T:main-intro} as a corollary. 

\begin{proof}[Proof of Theorem \ref{T:main}]
Let $\mu$ be a  totally irreducible and bi-contracting measure on $\Gamma$. Recalling that $(\overline{\Omega}, \overline{\mathbb{P}})$ is the space of bilateral infinite sample paths, we have defined in Eq. \eqref{E:double-boundary} a measurable map 
$$(\overline{\Omega}, \overline{\mathbb{P}}) \to (B \times B, \nu \otimes \check{\nu})$$ 
to the double Furstenberg boundary, where $B = G/P$ and $\nu$ and $\check{\nu}$ denote, respectively, the hitting measure of the forward and backward random walks. This shows that $(B, \nu)$ is a $\mu$-boundary for $(\Gamma, \mu)$, and we need to prove that it is maximal.

If we let $\gamma_{n}^{\alpha, L} := \tau_n^{\alpha,L} \vee \pi_n^{\alpha,L}  \vee  \sigma_n^{\alpha,L} \vee \beta_n^{\alpha,L}$, then the monotonicity properties of conditional entropy yield
$$\frac{1}{n} H_{B \times B}(\alpha_n) \leq \frac{1}{n} H_{B \times B}(\alpha_n \ | \ \gamma_n^{\alpha, L}) + \frac{1}{n} H(\gamma_n^{\alpha, L})$$
Now, by Proposition \ref{E:entropy-upper-bound}, for any $\epsilon >0$ there exist $\alpha, L >0$ such that 
$$\limsup_{n \to \infty} \frac{1}{n} H_{B \times B}(\alpha_n \ | \ \gamma_{n}^{\alpha, L}) \leq \epsilon$$
while  by Proposition \ref{P:pin-down} for any $\alpha, L >0$ we have
 $$\limsup_{n \to \infty} \frac{1}{n} H(\gamma_{n}^{\alpha, L}) =0$$
hence 
$$\lim_{n \to \infty} \frac{1}{n} H_{B \times B}(\alpha_n) = 0.$$
Then, noting that the backward and forward random walks are independent, and $\alpha_n$ depends only on the forward walk, we have 
$\mathbb{P}^{(b_-, b_+)}(A) = \mathbb{P}^{(b_+)}(A)$ for any $A \in \alpha_n$.
\begin{align*}
H_{B \times B}(\alpha_n) & = - \int_{B \times B} \sum_{A \in \alpha_n} \mathbb{P}^{(b_-, b_+)}(A) \log \mathbb{P}^{(b_-, b_+)}(A) \ d\check{\nu}(b_-) \ d\nu(b_+) \\
& = - \int_{B \times B} \sum_{A \in \alpha_n}  \mathbb{P}^{(b_+)}(A) \log \mathbb{P}^{(b_+)}(A) \ d\check{\nu}(b_-) \ d\nu(b_+) \\
& =  H_{B}(\alpha_n)
\end{align*}
hence also 
$$\lim_{n \to \infty} \frac{1}{n} H_{B}(\alpha_n) = 0$$
which implies by Theorem \ref{thm:maximal} that $(B, \nu)$ is the Poisson boundary. 
\end{proof}

\bibliographystyle{amsalpha}
\bibliography{ref}
\end{document}